\documentclass[12pt]{article}
\usepackage{amssymb}
\usepackage{amsmath}
\usepackage{endnotes}
\usepackage[doublespacing]{setspace}

\newtheorem{theorem}{Theorem}

\newenvironment{proof}[1][Proof]{\noindent\textbf{#1.} }{\ \rule{0.5em}{0.5em}}
\input{tcilatex}
\let\footnote=\endnote

\begin{document}

\title{Identification and well-posedness in a class of nonparametric problems%
}
\author{Victoria Zinde-Walsh\thanks{%
The support of the Social Sciences and Humanities Research Council of Canada
(SSHRC) and the \textit{Fonds\ qu\'{e}becois de la recherche sur la soci\'{e}%
t\'{e} et la culture} (FRQSC) is gratefully acknowledged. } \\
\\
McGill University and CIREQ\\
victoria.zinde-walsh@mcgill.ca\\
(514) 398 4834}
\maketitle
\date{}

\begin{center}
\bigskip \pagebreak

{\LARGE Abstract}
\end{center}

This is a companion note to Zinde-Walsh (2010) to clarify and extend results
on identification in a number of problems that lead to a system of
convolution equations. Examples include identification of the distribution
of mismeasured variables, of a nonparametric regression function under
Berkson type measurement error, some nonparametric panel data models, etc.
The reason that identification in different problems can be considered in
one approach is that they lead to the same system of convolution equations;
moreover the solution can be given under more general assumptions than those
usually considered, by examining these equations in spaces of generalized
functions. An important issue that did not receive sufficient attention is
that of well-posedness. This note gives conditions under which
well-posedness obtains, an example that demonstrates that when
well-posedness does not hold functions that are far apart can give rise to
observable arbitrarily close functions and discusses misspecification and
estimation from the stand-point of well-posedness.\ 

\section{\protect\bigskip Introduction}

The results of this paper apply to a number of econometric problems,
including the examples below.

\textbf{Example 1.} The distribution of a mismeasured variable with another
observation.

See, e.g., reviews of Carroll, Rupert and Stefanski (1995); Chen, Hong and
Nekipelov (2009); the problem is examined in Cunha, Heckman and Schennach
(2010).

Suppose that $g$ is the density of a mismeasured variable, $x^{\ast },$ $z$
is observed and has density $w_{1}$; $z=x^{\ast }+u,$ where $u$ is
measurement/contamination error independent of $x^{\ast }$ with a density, $%
f.$ Another observation, $x,$ on $x^{\ast }$ is available: $x=x^{\ast
}+u_{x},$ where $u_{x}$ is not necessarily independent but $E(u_{x}|x^{\ast
},u)=0.$%
\begin{eqnarray}
x &=&x^{\ast }+u_{x};  \label{measer1} \\
z &=&x^{\ast }+u.  \label{measer2}
\end{eqnarray}

\textbf{Example 2.} Errors in variables regression (EIV) model with Berkson
type measurement error.

Review Chen, Hong and Nekipelov (2009); examined by Newey (2001), univariate
case in Schennach (2007) and Zinde-Walsh (2009), multivariate Zinde-Walsh
(2010).

Consider 
\begin{eqnarray}
y &=&g(x^{\ast })+u_{y};  \label{a} \\
x &=&x^{\ast }+u_{x};  \label{b} \\
z &=&x^{\ast }+u.  \label{c}
\end{eqnarray}%
Here (\ref{a})-(\ref{c}) provide a regression with $z$ representing a second
measurement or possibly a given projection onto a set of instruments for the
unobserved $x^{\ast }$. Here $y,z$ or $x,y,z$ are observed; $u$ is a Berkson
type\ measurement error independent of $z$; $u_{y},u_{x}$ have zero
conditional (on $z$ and the other errors) expectations. Denote $%
w_{1}=E(y|z), $ density of measurement error $f.$

\textbf{Example 3.} Panel data model with two periods.

Evdokimov (2010).

Here let $x$ (or $z)$ represent the observed variable in the first period,
and $z$ ($x)$ for the second, $x^{\ast }$ is the nonparametric function $%
m(X,\alpha ),$ where $\alpha $ is the idiosyncratic component and the
densities are conditional on the same value $X$ for the two periods; the
same distributional assumptions as in Example 1 are used.

The models lead to the same system of convolution equations. All vectors are
in $R^{d}$.

By independence in all cases we get 
\begin{equation*}
g\ast f=w_{1}.
\end{equation*}

For examples 1 and 3 define density of $z$ by $f_{z};$ by $x_{k}$ the $kth$
component of the vector $x$ and consider

\begin{equation*}
E(f_{z}x_{k}|z)=E(x_{k}^{\ast }f_{z}|z)=\int
(z-u)_{k}g(z-u)f(u)du=x_{k}g\ast f.
\end{equation*}%
Denote the observable $E(f_{z}x_{k}|z)$ by $w_{2k},$ $k=1,...d.$

For example 2 
\begin{equation*}
E(x_{k}y|z)=E(x_{k}^{\ast }g(x^{\ast })|z)=\int (z_{k}-u_{k})g(z-u)f(u)du.
\end{equation*}%
Denote here $E(x_{k}y|z)$ by $w_{2k},$ $k=1,...d.$

Thus for all the examples we need to solve the system of convolution
equations%
\begin{eqnarray}
g\ast f &=&w_{1};  \label{syst} \\
x_{k}g\ast f &=&w_{2k},k=1,...d.  \notag
\end{eqnarray}

It is advantageous to consider the functions as generalized functions. The
interest is often in distributions of the unobservables and there is no
reason to restrict those to be absolutely continuous; density may not
necessarily exist but can be represented as a generalized derivative of the
distribution function rather than an ordinary function. Since solving the
convolution equations is done via Fourier transforms restricting regression
functions in Example 2 to have ordinary Fourier transforms excludes binary
choice or polynomial regression and can be overcome by using generalized
functions. Also,\ if some variables have singular distributions, or if only
some variables are subject to measurement error, or there is a mass point in
the error distribution (e.g. in measurement error from surveys with some
portion of true responses) convoluting with a generalized $\delta -$function
is natural when considering the problems in spaces of generalized functions.

The spaces of generalized functions most relevant for solving these problems
are the space $S^{\prime }$ (tempered distributions); space $D^{\prime }$
and some related spaces also play a role in the proofs. See e.g.
Zinde-Walsh, 2010 for the definitions, discussion and summary of useful
properties.

The next section 2 presents the full solution to system of equations $\left( %
\ref{syst}\right) $ extending\ all the results in the current literature.

Section 3 discusses well-posedness. This is to clarify two issues: under
what conditions consistent estimation of the identified general model is
possible and in what sense does a possibly mis-specified parametric model
deliver valid analysis. The answer hinges on well-posedness of the
identification of the function $g.$ Well-posedness refers to $g$ depending
on the distributions of the observed variables in a continuous fashion.
Well-posedness does not hold if both the function $g$ and the density $f$
are supersmooth (that is their Fourier transforms decline exponentially); on
the other hand if any one of the two is such that the Fourier transform is
continuously differentiable and its inverse is a regular generalized
function (grows no faster than some power), then well-posedness in the weak
topology of generalized functions obtains; for well-posedness in stronger
topologies additional conditions need to be provided. An example shows that
a Gaussian density for both functions would lead to violation of
well-posedness. Classes of nonparametric models that include the Gaussian
and that lead to a well-posed problem are defined. Further, the issue of
regularizattion is discussed.

\section{The identification result}

\textbf{Assumption 1. }The generalized functions $g,f,w_{1}$ and $w_{2k},$ $%
k=1,...,d,$ are in the generalized function space $S^{\prime }$ and are
related by $\left( \ref{syst}\right) .$

Any generalized density functions are generalized derivatives of the
distribution function and belong to $S^{\prime },$ convolution equations are
defined. For a ordinary function, $b,$ e.g. a regression function of example
2 to belong to $S^{\prime }$ it is sufficient that it belong to some class
of functions on $R^{d},$ $\Phi (m,V)$ (with $m$ a vector of integers, $V$ a
positive constant) where $b\in \Phi (m,V)$ if 
\begin{equation}
\int \Pi \left( (1+t_{i}^{2})^{-1}\right) ^{m_{i}}\left\vert b(t)\right\vert
dt<V<\infty .\text{ }  \label{condb}
\end{equation}%
Thus if e.g. $b$ grows no faster than a polynomial, it is in $S^{\prime },$
so that the analysis here applies to binary choice and polynomial
regression. Convolutions with generalized functions from some classes are
defined for such functions (as discussed in Zinde-Walsh, 2010). For
conditional density of Example 3 some extra assumptions on the joint density
of the regressors are required.

Consider now Fourier transforms ($Ft):$ $\ \gamma =Ft(g);\phi
=Ft(f);\varepsilon _{.}=Ft(w_{.}).$

\textbf{Assumption 2}. Either $\phi $ or $\gamma $ is a continuous function
such that it satisfies $\left( \ref{condb}\right) .$

The continuity assumption on the characteristic function is typically made;
any characteristic function satisfies$\left( \ref{condb}\right) .$

By Theorem 1 of Zinde-Walsh 2010 then the following system of equations
holds in $S^{\prime }.$

\begin{eqnarray}
\gamma \cdot \phi &=&\varepsilon _{1};  \label{system} \\
\gamma _{k}^{\prime }\cdot \phi &=&\varepsilon _{2k},\text{ }k=1,...d. 
\notag
\end{eqnarray}

\textbf{Assumption 3}. supp($\phi )\supseteq $supp($\gamma )=W,$\ where $W$
is a convex set in $R^{d}$ that includes an interior point $0.$

The support assumption is necessary to solve for $\gamma .$ The interior
point in the case of characteristic functions is zero with the value of the
continuous characteristic function equal to 1 at that point. If the system
of equations involves functions with $W$ having an interior point $a\neq 0$
consider shifted functions.

\begin{theorem}
Under Assumptions 1-3 if

(a) $\gamma $ is continuously differentiable in $W,$ $\gamma (0)=c$%
\begin{equation}
\gamma (\zeta )=c\exp \int_{0}^{\zeta }\Sigma _{k=1}^{d}\varkappa _{k}(\xi
)d\xi _{k},  \label{sola}
\end{equation}%
with the uniquely defined continuous functions $\varkappa _{k}(\xi )$ that
solve 
\begin{equation*}
\varkappa _{k}(\xi )\varepsilon _{1}-i\varepsilon _{2k}=0,k=1,...,d;
\end{equation*}
or

(b) $\phi $ is continuously differentiable in $W$, $\phi (0)=c$%
\begin{equation}
\gamma (\zeta )=\tilde{\phi}(\zeta )^{-1}\varepsilon _{i}(\zeta ),
\label{solb}
\end{equation}%
where 
\begin{equation*}
\tilde{\phi}(\zeta )=\exp \int_{0}^{\zeta }\Sigma _{k=1}^{d}\tilde{\varkappa}%
_{k}(\xi )d\xi _{k},
\end{equation*}%
with the uniquely defined continuous functions $\varkappa _{k}(\xi )$ that
solve 
\begin{equation*}
\varepsilon _{1}\tilde{\varkappa}_{k}-(\left( \varepsilon _{1}\right)
_{k}^{\prime }-i\varepsilon _{2k})=0,k=1,...,d.
\end{equation*}
\end{theorem}

The proof is in proof of Theorem 3 and the Corollary of Zinde-Walsh, 2010.
If support of $\phi $ coincides with support of $\gamma ,$ $\phi $ (and thus
the function $f)$ is identified.

The proof is set in the space $S^{\prime }$ of generalized functions and
does not rely on existence of densities. The proof of (b) in the univariate
case was first provided in Zinde-Walsh 2009, correcting the result of
Schennach 2007. The formula for the density in Cunha et al. 2010 is valid in
case (a) and can be interpreted in terms of generalized functions
("distributions"). In case (b), though, a different solution is given here.
Thus identification requires differentiability of either $\gamma $ or $\phi $%
; when $\gamma $ is not differentiable and the result in Cunha et al does
not hold identification is still possible in case (b). This also extends the
identification result of Evdokimov 2010.

\section{Well-posedness}

We now consider whether when the distributions of the observables are close
the unknown functions are also necessarily close.

A sufficient condition is provided in Zinde-Walsh 2010 (Theorem 4). When
identification is based on (b) of Theorem 1 here the model class\ needs to
be restricted to include only measurement error distributions with $\phi
^{-1}$ in $\Phi (m,V)$ for some $m,V.$ Equivalently, when identification is
based on (a), the sufficient condition is for the class of models to be
restricted to those where the latent factor distribution is such that $%
\gamma ^{-1}\in \Phi (m,V)$.

These conditions exclude models where both $g$ and $f$ are supersmooth with
supp($\gamma )$ unbounded leading to a supersmooth distribution for $w_{1}.$
Although these conditions are only shown to be sufficient, an example below
(from Zinde-Walsh 2009 and 2010) demonstrates that a Gaussian distribution
(that violates these conditions) fails well-posedness in the weak topology
of generalized functions in $S^{\prime }$ and therefore in any stronger
topology or metric (uniform, $L_{1},etc.).$

\textbf{Example 4.} \textit{Consider the function }$\phi (x)=e^{-x^{2}},$%
\textit{\ }$x\in R.$\textit{\ Consider in }$S$\textit{\ the function }$%
b_{n}(x)=$

\begin{equation}
\left\{ 
\begin{array}{cc}
e^{-n} & \text{if }n-\frac{1}{n}<x<n+\frac{1}{n}; \\ 
0<b_{n}(x)\leq e^{-n} & \text{if }n-\frac{2}{n}<x<n+\frac{2}{n}; \\ 
0 & \text{otherwise.}%
\end{array}%
\right.  \label{bn}
\end{equation}%
\textit{This }$b_{n}(x)$\textit{\ converges to }$b(x)\equiv 0$\textit{\ in }$%
S^{\prime }$\textit{. Indeed for any }$\psi \in S$\textit{\ }%
\begin{equation*}
\int_{-\infty }^{\infty }b_{n}(x)\psi (x)dx=\int_{n-2/n}^{n+2/n}b_{n}(x)\psi
(x)dx\rightarrow 0.
\end{equation*}

\textit{Now consider }$\varepsilon _{n}=\allowbreak \varepsilon
+b_{n}\rightarrow \varepsilon .$\textit{\ We show that }$\varepsilon
_{n}\phi ^{-1}$\textit{\ does not converge in }$S^{\prime }$\textit{\ to }$%
\varepsilon \phi ^{-1}.$\textit{\ Such convergence would imply that (}$%
\varepsilon _{n}-\varepsilon )\phi ^{-1}=b_{n}\phi ^{-1}\rightarrow 0$%
\textit{\ in }$S^{\prime }.$

\textit{But the sequence }$b_{n}(x)\phi (x)^{-1}$\textit{\ does not
converge. Indeed if it did then }$\int b_{n}(x)\phi ^{-1}(x)\psi (x)dx$%
\textit{\ would converge for any }$\psi \in S.$\textit{\ But for }$\psi \in
S $\textit{\ such that }$\psi (x)=\exp (-\left\vert x\right\vert )$\textit{\ 
}%
\begin{eqnarray*}
\int b_{n}(x)e^{x^{2}}\psi (x)dx &\geq
&\int_{n-2/n}^{n+2/n}b_{n}(x)e^{x^{2}}\psi (x)dx\geq
e^{-n}\int_{n-1/n}^{n+1/n}e^{x^{2}-x}dx \\
&\geq &\frac{2}{n}e^{-2n+\left( n-1/n\right) ^{2}}.
\end{eqnarray*}%
\textit{This diverges}.$\blacksquare $

Thus, e.g. for the Gaussian distribution there are models with unknown
functions that are far from each other in $S^{\prime }$, but that lead to
observable functions that are arbitrarily close.

When the nonparametric identification result is interpreted to support
possible wider applicability, when estimation is in fact based on a
parametric model the question arises as to which nonparametric models are
close to a model misspecified as parametric. This question may be posed e.g.
for the analysis of Cunha et al 2010 who use Gaussian and mixed Gaussian
distributions in estimation. Is there some meaningful nonparametric class
that includes the Gaussian where observationally close models imply
closeness of latent factors?

Define a class of generalized functions $\Phi (B,\Lambda ,m,V)\subset
S^{\prime }$ for some positive constant $B$ and matrix $\Lambda $; a
generalized function $b\in \Phi (B,\Lambda ,m,V)$ if there exists a function 
$\bar{b}(\zeta )\in \Phi (m,V)$ with support in $\left\Vert \zeta
\right\Vert >B$ such that also $\bar{b}(\zeta )^{-1}\in \Phi (m,V)$ and $%
b\cdot I(\left\Vert \zeta \right\Vert >B)=\bar{b}(\zeta )\exp \left( -\zeta
^{\prime }\Lambda \zeta \right) .$ Note that a linear combination of
functions in $\Phi (B,\Lambda ,m,V)$ belongs to the same class. For a
sequence of $b_{n}\in \Phi (B,\Lambda ,m,V)$ to converge to zero as
generalized functions it is necessary that the corresponding $\bar{b}_{n}$
converge to zero (a.e.).

\textbf{Assumption 4.} $\gamma \in \Phi (B,\Lambda _{\gamma },m,V);$ $\phi
\in \Phi (B,\Lambda _{\phi },m,V).$

If this assumption holds $\varepsilon _{1}=\gamma \cdot \phi \in \Phi
(B,\Lambda _{\gamma }+\Lambda _{\phi },2m,V^{2})$ and $\varepsilon
_{2}=\gamma _{k}^{\prime }\cdot \phi $ also is in $\Phi (B,\Lambda _{\gamma
}+\Lambda _{\phi },2m,V^{2}).$

\begin{theorem}
Under conditions of Theorem 1 and the Assumption 4 applying to generalized
functions $\varepsilon _{i,n},$ $i=1,2$ if $\varepsilon _{i,n}\rightarrow
\varepsilon _{i}$ in $S^{\prime },$ then the corresponding solutions $\gamma
_{n}$ given by (\ref{sola}) or (\ref{solb}) converge to $\gamma $ in $%
S^{\prime }.$
\end{theorem}

\begin{proof}
From the conditions of the theorem $\eta _{i,n}=\varepsilon
_{i,n}-\varepsilon _{i}$ converges in $S^{\prime }$ to zero. From the nature
of the identified solution (\ref{sola}) or (\ref{solb}) it follows that if
it can be shown that $\varepsilon _{1}^{-1}\eta _{\cdot ,n}$ converges to
zero then the problem for the distribution of the latent factor is
well-posed. But this is indeed the case since the exponents cancel, $\bar{%
\eta}_{\cdot ,n}$ converges to zero in $S^{\prime }$ and convergence to zero
follows by hypocontinuity of the product $\bar{\varepsilon}_{1}^{-1}\bar{\eta%
}_{\cdot ,n}$.
\end{proof}

The values for the tail exponent have to be fixed implying that even a
slight deviation in the exponent violates the well-posedness condition:
there are no two different Gaussian distributions in the class. Since an
estimated Gaussian will differ from the true Gaussian they cannot belong to
the same non-parametric class thus there is separation between estimation in
the parametric Gaussian problem and\ in the fully general nonparametric
specification in $S^{\prime }.$

Consider a solution regularized with a weighting function. It is high
frequency components that cannot be identified in convolution with a super
smooth function and regularization smooths those out. Fix a function $\psi
\in D.$ Using this weight on the Fourier transform is equivalent to solving
convolution equations%
\begin{eqnarray}
g\ast f\ast (Ft^{-1}(\psi )) &=&w_{1}\ast (Ft^{-1}(\psi ));  \label{weight}
\\
x_{k}g\ast f\ast (Ft^{-1}(\psi )) &=&w_{2k}\ast (Ft^{-1}(\psi )),k=1,...d. 
\notag
\end{eqnarray}

As in the proof of Theorem 3 of Zinde-Walsh 2010 for any such $\psi $ the
solution exists because multiplication by a continuous function $\gamma
^{-1} $ or $\phi ^{-1}$ with arbitrary growth at infinity is permitted since
support of $\psi $ is bounded.

Schwatz (1964, pp.271-273) gives a characterization of functions in $%
S^{\prime }$ with Fourier transform that has bounded support (in a cube $%
\left\vert x_{k}\right\vert <C,$ $k=1,...,d)$ based on Wiener-Paley theorem.
Such a function is a continuous function $g$ that can be extended to a
entire analytic function $G$ of a complex argument and is of exponential
type $\leq 2\pi C,$ meaning%
\begin{equation*}
\underset{\left\vert z\right\vert \rightarrow \infty }{\lim }\sup \frac{\log
\left\vert G(z)\right\vert }{\left\vert z_{1}\right\vert +...+\left\vert
z_{d}\right\vert }\leq 2\pi C.
\end{equation*}%
Thus as long as $g$ is such a function it can be expressed via the
regularized solution. As in Schwartz the subspace of all functions of
exponential type (for any finite $C)$ can also be considered. However, the
regularized solutions may not come close to a true $g$ that does not belong
to this subspace.

\end{document}